\documentclass[10pt,letterpaper]{article}
\usepackage[utf8]{inputenc}
\usepackage[english]{babel}
\usepackage{amsmath}
\usepackage{amsfonts}
\usepackage{amssymb}
\usepackage{graphicx}
\usepackage{mathrsfs}
\usepackage{upref,amsthm,amsxtra,exscale}
\usepackage{cite}
\usepackage[colorlinks=true,urlcolor=blue,
citecolor=red,linkcolor=blue,linktocpage,pdfpagelabels,
bookmarksnumbered,bookmarksopen]{hyperref}

\newtheorem{theorem}{Theorem}[section]
\newtheorem{corollary}[theorem]{Corollary}
\newtheorem{lemma}[theorem]{Lemma}
\newtheorem{proposition}[theorem]{Proposition}

\newtheorem{example}[theorem]{Example}

\numberwithin{equation}{section}

\author{Mónica Clapp\footnote{M. Clapp was partially supported by UNAM-DGAPA-PAPIIT grant IN100718 (Mexico), CONACYT grant A1-S-10457 (Mexico).}, Angela Pistoia and Tobias Weth}
\title{An upper bound for the least energy of a nodal solution to the Yamabe equation on the sphere}
\date{\today}

\begin{document}

\maketitle

\begin{abstract}

For each $n\geq 3$ we establish the existence of a nodal solution $u$ to the Yamabe problem on the round sphere $(\mathbb{S}^n,g)$ which satisfies
$$\int_{\mathbb{S}^n}|u|^{2^*}dV_g < 2m_n\mathrm{vol}(\mathbb{S}^n),$$
where $m_3=9,$ $m_4= 7,$ $m_5=m_6=6$, and  $m_n= 5\ \text{if }n\geq 7.$ \medskip


\end{abstract}

\section{Introduction} \label{sec:introduction}

We consider the Yamabe problem
\begin{equation} \label{eq:1}
\frac{4(n-1)}{n-2}\Delta_g u + n(n-1)u= n(n-1)|u|^{2^*-2}u\qquad\text{on }\mathbb{S}^n,
\end{equation}
 on the round $n$-sphere $(\mathbb{S}^n,g)$, $n\geq 3$, where $\Delta_g u=-\mathrm{div}_g\nabla_g$ is the Laplace-Beltrami operator and $2^*:=\frac{2n}{n-2}$ is the critical Sobolev exponent. 

The existence of positive and sign-changing solutions to this problem is well known. Different types of nodal solutions have been exhibited in \cite{c,d,dmpp,fp}. 

It is easily observed that any nodal solution $u$ of \eqref{eq:1} satisfies
\begin{equation}
  \label{eq:struwe-nodal-est}
\int_{\mathbb{S}^n}|u|^{2^*}dV_g >2\mathrm{vol}(\mathbb{S}^n),  
\end{equation}
see e.g. \cite[Chapter III.3]{struwe}. This estimate has been slightly improved in \cite{w}, where it has been proved that 
\begin{equation}
  \label{eq:weth-est}
\inf \Bigl \{\:\int_{\mathbb{S}^n}|u|^{2^*}dV_g \::\: \text{$u$ nodal solution of (\ref{eq:1})} \: \Bigr\} >2\mathrm{vol}(\mathbb{S}^n).
\end{equation}
We note that (\ref{eq:weth-est}) is not a direct consequence of (\ref{eq:struwe-nodal-est}), since it is unknown if the infimum in (\ref{eq:weth-est}) is attained.
 
Estimates for the least energy of nodal solutions to problem \eqref{eq:1} are of interest, since they are related to compactness properties of semilinear elliptic boundary value problems with critically growing nonlinearities via Struwe's compactness lemma. See \cite[Chapter III.3]{struwe} for a discussion of this aspect.    

The aim of this note is to give an upper bound for the least energy of a nodal solution to problem \eqref{eq:1}. Set
\begin{equation} \label{eq:m(n)}
m_n:=
\begin{cases}
9 &\text{if }n=3,\\
7 &\text{if }n=4,\\
6 &\text{if }n=5,6,\\
5 &\text{if }n\geq 7.
\end{cases}
\end{equation}
We prove the following result.

\begin{theorem} \label{thm:main}
The Yamabe equation \eqref{eq:1} has a nodal solution $u$ which satisfies
\begin{align*}
u(z_1,z_2,x) &= u(\mathrm{e}^{2\pi\mathrm{i}/m_n}z_1,\mathrm{e}^{2\pi\mathrm{i}/m_n}z_2,x)\\
u(z_1,z_2,x) &= -u(-\bar{z}_2,\bar{z}_1,x),
\end{align*}
for all $(z_1,z_2,x)\in\mathbb{C}\times\mathbb{C}\times\mathbb{R}^{n-3}\equiv\mathbb{R}^{n+1}$, and
$$\int_{\mathbb{S}^n}|u|^{2^*}dV_g < 2m_n\mathrm{vol}(\mathbb{S}^n).$$
\end{theorem}

The solution given by Theorem \ref{thm:main} might be the same as the one obtained in \cite[Theorem 1.1]{c} for $n\geq 4$, but it is different from those obtained by Ding in \cite{d}, as shown in Proposition \ref{prop:noding} below. Estimates for the energy of some of Ding's solutions are listed in \cite{fp}, but no information is given which allows to verify them.

Our approach is as follows: First, we give a condition for the existence of a least energy solution to the Yamabe problem \eqref{eq:1} with a particular type of symmetries (see Corollary \ref{cor:existence}). The symmetries are chosen in such a way that they yield sign-changing solutions by construction. Then, we estimate the energy of a specific ansatz  and derive an explicit condition on the symmetries which guarantees the validity of the requirement \eqref{eq:simple} in Corollary \ref{cor:existence} (see Proposition \ref{prop:mu}). Finally, we prove that the condition on the symmetries holds true for the particular example that gives rise to Theorem \ref{thm:main}.

\subsection*{Acknowledgement}

The third author wishes to thank Farid Madani for valuable contributions. Moreover, he wishes to thank Andrea Malchiodi and Fr\'ed\'eric Robert for helpful discussions on the topic.

\section{Symmetric nodal solutions}

The group $O(n+1)$ of linear isometries of $\mathbb{R}^{n+1}$ acts isometrically on $\mathbb{S}^n$. We fix a closed subgroup $\Gamma$ of $O(n+1)$ and, as usual, we denote by
$$\Gamma p:=\{\gamma p:\gamma\in\Gamma\}\qquad\text{ and }\qquad\Gamma_p:=\{\gamma\in\Gamma:\gamma p=p\}$$
the $\Gamma$-orbit and the $\Gamma$-isotropy subgroup of a point $p$ in $\mathbb{S}^n$. Recall that $\Gamma p$ is $\Gamma$-diffeomorphic to the homogeneous space $\Gamma/\Gamma_p$. So, they have the same cardinality, i.e., $\#\Gamma p=|\Gamma/\Gamma_p|$, the index of $\Gamma_p$ in $\Gamma$.

Let $\phi:\Gamma\to\mathbb{Z}_2:=\{1,-1\}$ be a continuous homomorphism of groups. We shall look for solutions $u:\mathbb{S}^{n}\to\mathbb{R}$ to the Yamabe equation \eqref{eq:1} which satisfy
\begin{equation} \label{eq:equivariant}
u(\gamma p)=\phi(\gamma)u(p)\qquad\forall\gamma\in\Gamma,\quad p\in \mathbb{S}^{n}. 
\end{equation}
A function $u$ with this property will be called $\phi$\emph{-equivariant}. It might occur that the only function $u$ satisfying \eqref{eq:equivariant} is the trivial function. This happens, e.g., if $\Gamma=O(n+1)$ and $\phi(\gamma)$ is the determinant of $\gamma$. To avoid this bad behavior, we will assume, from now on, that $\phi$ satisfies the following assumption
\begin{itemize}
\item[$(A_0)$] There exists $p_0\in\mathbb{S}^{n}$ such that $\Gamma_{p_0}\subset\ker\phi=:G$.
\end{itemize}
This assumption guarantees that the space
$$H^1_g(\mathbb{S}^{n})^{\phi}:=\{u\in H^1_g(\mathbb{S}^{n}):u\text{ is }\phi\text{-equivariant}\}$$
is infinite dimensional; see \cite{bcm}.

If $\phi\equiv 1$, then \eqref{eq:equivariant} simply says that $u$ is a $\Gamma$-invariant function. On the other hand, if $\phi$ is surjective and $u$ is nontrivial, then \eqref{eq:equivariant} implies that $u$ is sign-changing and $G$-invariant, where $G=\ker\phi$.

Set $a_n:=\frac{n(n-2)}{4}$. We take
\begin{equation}\label{eq:norms}
\|u\|:=\left(\int_{\mathbb{S}^{n}}\left[|\nabla_{g}u|_{g}^{2}+a_nu^{2}\right]dV_{g}\right)^\frac{1}{2},\quad|u|_{2^*}:=\left(\int_{\mathbb{S}^{n}}a_n|u|^{2^{\ast}}dV_{g}\right)^\frac{1}{2^*}
\end{equation}
as the norms in $H^1_g(\mathbb{S}^{n})$ and $L^{2^*}_g(\mathbb{S}^{n})$, respectively.

The $\phi$-equivariant solutions to the Yamabe equation \eqref{eq:1} are the critical points of the functional $J_n:H^1_g(\mathbb{S}^{n})^{\phi}\to\mathbb{R}$ given by
$$J_n(u)=\frac{1}{2}\|u\|^{2}-\frac{1}{2^*}|u|_{2^*}^{2^*}.$$
The nontrivial ones lie on the Nehari manifold
$$\mathcal{N}^{\phi}(\mathbb{S}^{n}):=\{u\in H_{g}^{1}(\mathbb{S}^{n})^{\phi}:u\neq0,\ \|u\|^{2}=|u|_{2^*}^{2^*}\}.$$
Set
$$c^{\phi}_n:=\inf_{u\in\mathcal{N}^{\phi}(\mathbb{S}^{n})}J_n(u).$$
Assumption $(A_0)$ implies that $\mathcal{N}^{\phi}(\mathbb{S}^{n})\neq\emptyset$. Therefore, $c^{\phi}_n\in\mathbb{R}$.
 
If $\phi\equiv 1$, then $u\equiv 1$ belongs to $\mathcal{N}^{\phi}(\mathbb{S}^{n})$ and minimizes $J_n$ on $\mathcal{N}^{\phi}(\mathbb{S}^{n})$. Hence,
\begin{equation} \label{eq:trivial}
c^{\phi}_n=\frac{n-2}{4}\mathrm{vol}(\mathbb{S}^{n})=:c_n\qquad\text{if }\phi\equiv 1.
\end{equation}
If $K$ is a closed subgroup of $\Gamma$ we write $\phi|K$ for the restriction of the homomorphism $\phi$ to $K$. Then we have that
$$c^{\phi}_n\geq c^{\phi|K}_n\geq c^{\phi|\{1\}}_n=c_n.$$

The following result gives conditions for the existence of a minimizer.

\begin{theorem} \label{thm:existence}
Assume $(A_0)$ holds true and there exists $q\in\mathbb{S}^{n}$ with $\Gamma_q=\Gamma$ (i.e., $q$ is a $\Gamma$-fixed point). If
\begin{equation} \label{eq:involved}
c^{\phi}_n<\min\{(\#\Gamma p)\,c^{\phi|\Gamma_p}_n:p\in\mathbb{S}^{n}\text{ and }\Gamma_p\neq\Gamma\},
\end{equation}
then there exists $u\in\mathcal{N}^{\phi}(\mathbb{S}^{n})$ such that $J_n(u)=c^{\phi}_n$,\, i.e., the Yamabe problem \eqref{eq:1} has a nontrivial least energy $\phi$-equivariant solution. This solution changes sign if $\phi$ is surjective.
\end{theorem}

\begin{proof}
After a change of coordinates, we may assume that $q=(0,\ldots,0,1)$. Then, $\Gamma$ acts trivially on the second factor of $\mathbb{R}^n\times\mathbb{R}\equiv\mathbb{R}^{n+1}$ and the stereographic projection from the point $q$ induces an orthogonal action of $\Gamma$ on $\mathbb{R}^n$.  It is well known that the Yamabe problem \eqref{eq:1} on the round sphere is equivalent to the problem
$$-\Delta v = a_n|v|^{2^*-2}v,\qquad v\in D^{1,2}(\mathbb{R}^n),$$
via the stereographic projection. So the statement follows from \cite[Theorem 3.3]{c}.
\end{proof}

Theorem \ref{thm:existence} is also true when $\mathbb{S}^{n}$ does not contain a $\Gamma$-fixed point, but it cannot be derived from \cite[Theorem 3.3]{c} and the proof requires some work. 

The symmetries we shall consider in this paper satisfy the following additional assumptions. We denote by $1$ the identity in $O(n+1)$.

\begin{itemize}
\item[$(A_1)$] Either $\Gamma_p=\Gamma$, or $\Gamma_p=\{1\}$, for any $p\in\mathbb{S}^{n}$.
\item[$(A_2)$] $\phi:\Gamma\to\mathbb{Z}_2$ is surjective.
\end{itemize}

Under these assumptions, condition \eqref{eq:involved} becomes considerably simpler, and a standard argument allows to extend Theorem \ref{thm:existence} to the case when $\mathbb{S}^{n}$ does not contain a $\Gamma$-fixed point. Namely, we have the following result.

\begin{corollary} \label{cor:existence}
Assume $(A_0)$, $(A_1)$ and $(A_2)$. If
\begin{equation} \label{eq:simple}
c^{\phi}_n <(\#\Gamma)\,c_n=\frac{n-2}{4}(\#\Gamma)\,\mathrm{vol}(\mathbb{S}^{n}),
\end{equation}
then the Yamabe problem \eqref{eq:1} has a nontrivial least energy $\phi$-equivariant solution. This solution changes sign.
\end{corollary}

\begin{proof}
If $\Gamma_p=\Gamma$ for every $p\in\mathbb{S}^n$, then $(A_0)$ implies $\phi\equiv 1$, contradicting $(A_2)$. So, by $(A_1)$, the right-hand side of \eqref{eq:involved} is $(\#\Gamma)\,c_n$ and the statement follows from Theorem \ref{thm:existence} if $\mathbb{S}^{n}$ contains a $\Gamma$-fixed point.

If $\mathbb{S}^{n}$ does not contain a $\Gamma$-fixed point, then $\Gamma$ acts freely on $\mathbb{S}^{n}$ and the same argument given to prove \cite[Theorem 2.2]{cf} yields this result.
\end{proof}

An immediate consequence of Corollary \ref{cor:existence} is the following fact.

\begin{corollary} \label{cor:infinite_orbits}
Assume $(A_0)$, $(A_1)$ and $(A_2)$. If $\#\Gamma=\infty$, then the Yamabe problem \eqref{eq:1} has a nontrivial least energy $\phi$-equivariant solution. This solution changes sign.
\end{corollary}

Next, we give some examples. We write $\mathbb{R}^{n+1}\equiv\mathbb{C}\times\mathbb{C}\times\mathbb{R}^{n-3}$, and the points in $\mathbb{R}^{n+1}$ as $(z_1,z_2,x)$ with $z_i\in\mathbb{C}$ and $x\in\mathbb{R}^{n-3}$. 

\begin{example} \label{ex:m}
For $m\geq 1$ let $G_m:=\{\mathrm{e}^{2\pi\mathrm{i}j/m}:j=0,\ldots,m-1\}$, $\Gamma_m$ be the group generated by $G_m\cup\{\tau\}$, acting on $\mathbb{R}^{n+1}$ as 
\begin{equation*}
\mathrm{e}^{2\pi\mathrm{i}j/m}(z_1,z_2,x):=(\mathrm{e}^{2\pi\mathrm{i}j/m}z_1,\mathrm{e}^{2\pi\mathrm{i}j/m}z_2,x),\qquad \tau(z_{1},z_{2},x):=(-\bar{z}_{2},\bar{z}_{1},x),
\end{equation*}
and $\phi_m:\Gamma_m\to\mathbb{Z}_{2}$ be the homomorphism given by $\phi_m(\mathrm{e}^{2\pi\mathrm{i}j/m}):=1$ and $\phi_m(\tau):=-1$. For any $p=(z_1,z_2,x)$, the $\Gamma_m$-orbit of $p$ is
    \begin{equation*}
	\Gamma_m p = 
	\begin{cases}
	\{p\} & \text{if }z_1=z_2=0, \\ 
	G_mp \,\sqcup\, G_m(\tau p) & \text{if either } z_1\neq 0,\text{ or }z_2\neq 0.
	\end{cases}
	\end{equation*}
Assumptions $(A_0)$, $(A_1)$ and $(A_2)$ are clearly satisfied.
\end{example}

\begin{example} \label{ex:infty}
Let $\Gamma_\infty$ be the group generated by $\{\mathrm{e}^{\mathrm{i}\vartheta}: \vartheta \in [0,2\pi)\}\cup\{\tau\}$, acting on $\mathbb{R}^{n+1}$ as
\begin{equation*}
\mathrm{e}^{\mathrm{i}\vartheta}(z_1,z_2,x):=(\mathrm{e}^{\mathrm{i}\vartheta}z_1,\mathrm{e}^{\mathrm{i}\vartheta}z_2,x),\qquad \tau(z_{1},z_{2},x):=(-\bar{z}_{2},\bar{z}_{1},x),
\end{equation*}
and let $\phi_\infty :\Gamma_\infty\to\mathbb{Z}_{2}$ be the homomorphism given by $\phi_\infty(\mathrm{e}^{\mathrm{i}\vartheta}):=1$ and $\phi_\infty(\tau):=-1$. Then, the assumptions of \emph{Corollary} \ref{cor:infinite_orbits} are satisfied. Hence, the Yamabe problem \eqref{eq:1} has a nontrivial least energy $\phi_\infty$-equivariant solution, which changes sign. For $n\geq 4$ this solution was exhibited in \emph{\cite[Theorem 1.1]{c}}.
\end{example}

For $\phi_m:\Gamma_m\to\mathbb{Z}_{2}$ as in Example \ref{ex:m}, Corollary \ref{cor:existence} yields the existence of a sign-changing solution to \eqref{eq:1}, whose energy is $c^{\phi_m}_n$, if the inequality
\begin{equation} \label{eq:m}
c^{\phi_m}_n<2m\,c_n
\end{equation}
is satisfied. As $\Gamma_m$ is a subgroup of $\Gamma_\infty$, we have that $c^{\phi_m}_n\leq c^{\phi_\infty}_n$ for all $m\in\mathbb{N}$. So \eqref{eq:m} holds true for sufficiently large $m$. On the other hand, as shown in \cite{w}, the least energy of a sign-changing solution to \eqref{eq:1} is strictly larger than $2c_n$, so \eqref{eq:m} is not satisfied for $m=1$.

In the next section we estimate the smallest $m$ for which \eqref{eq:m} holds true.

\section{Estimates for the energy of nodal solutions}

Let $\Gamma$ be a finite subgroup of $O(n+1)$ and $\phi:\Gamma\to\mathbb{Z}_{2}$ be a homomorphism satisfying $(A_0)$, $(A_1)$ and $(A_2)$. Fix $\hat{\gamma}\in\Gamma$ with $\phi(\hat{\gamma})=-1$ and write
\begin{equation*}
\Gamma=\{g_1,\ldots,g_m,\hat{\gamma}g_1,\ldots,\hat{\gamma}g_m\}\quad\text{with }g_1=1,\;\phi(g_i)=1\text{ for }i=1,\ldots,m.
\end{equation*}
The $\Gamma$-orbit of $p\in\mathbb{S}^n$ is
$$\Gamma p=\{p_1,\ldots,p_m,q_1,\ldots,q_m\}\qquad\text{with}\quad p_j:=g_jp\quad\text{and}\quad q_j:=\hat{\gamma}p_j.$$
We set
$$\mu_p:=\sum_{1\leq i\neq j\leq k}(1-\cos\mathrm{d}_g(p_i,p_j))^\frac{2-n}{2},\qquad\widehat{\mu}_p:=\sum_{1\leq i, j\leq k} (1-\cos \mathrm{d}_g(p_i,q_j))^{\frac{2-n}{2}},$$
where $\mathrm{d}_g(p,p')=\arccos\langle p,p'\rangle$ is the geodesic distance from $p$ to $p'$ on $\mathbb{S}^n$.

\begin{proposition}\label{prop:mu}
If $\mu_p - \widehat{\mu}_p>0$ for some $p\in\mathbb{S}^n$, then $c^{\phi}_n<2m\,c_n$.
\end{proposition}

To prove this proposition, we fix $p\in\mathbb{S}^n$ and, for each $\beta>1$, we define
\begin{equation*}
u_{\beta}(q):=(\beta^2-1)^{\frac{n-2}{4}}(\beta-\cos\mathrm{d}_g(p,q))^{-\frac{n-2}{2}},\qquad q\in\mathbb{S}^n.
\end{equation*}
The function $u_{\beta}$ is a positive least energy solution of the Yamabe equation \eqref{eq:1}. Hence,
$$J_n(u_\beta)=\frac{1}{n}\|u_\beta\|^2=\frac{1}{n}|u_\beta|^{2^*}=c_n=\frac{n-2}{4}\mathrm{vol}(\mathbb{S}^n).$$
We denote by $B_\delta(p)$ the geodesic ball of radius $\delta$ centered at $p$ in $\mathbb{S}^n$, and set $\omega_n:=\mathrm{vol}(\mathbb{S}^n)$.

\begin{lemma} \label{lem:estimate}
For any $f\in\mathcal{C}^0(\mathbb{S}^n)$ and $\delta\in (0,\pi)$, we have
\begin{align*}
&\int_{B_\delta(p)}fu_{\beta}^{2^*-1}\mathrm{d}V_g = \frac{2^{\frac{3n+2}{4}}\omega_{n-1}}{n}f(p)(\beta-1)^{\frac{n-2}{4}}+ o\left((\beta-1)^{\frac{n-2}{4}}\right)\quad\text{as }\beta\to 1,\\
&\int_{\mathbb{S}^n\smallsetminus B_\delta(p)}fu_{\beta}^{2^*-1}\mathrm{d}V_g = O\left((\beta-1)^{\frac{n+2}{4}}\right)\quad\text{for }\beta \text{ close to }1.
\end{align*}
\end{lemma}

\begin{proof}
Let $\sigma:\mathbb{S}^n\smallsetminus\{p\}\to\mathbb{R}^n$ be the stereographic projection. Then,
$$|x|=\cot\frac{r}{2}\quad\text{ and }\quad r=\arccos\left(\frac{|x|^2-1}{|x|^2+1}\right),\quad\text{ if }x=\sigma(q),\;r:=\mathrm{d}_g(p,q).$$
The pullback of the round metric in the local coordinates $\sigma^{-1}:\mathbb{R}^n\to\mathbb{S}^n\smallsetminus\{p\}$ is $(\sigma^{-1})^*g=\frac{4}{(1+|x|^2)^2}\bar{g}$, where $\bar{g}$ is the Euclidean metric. Writing $x=\sqrt{\frac{\beta +1}{\beta -1}}\frac{\xi}{\rho}$ with $\xi\in\mathbb{S}^{n-1}$, we obtain
\begin{align*}
&\int_{B_\delta(p)}fu_{\beta}^{2^*-1}\mathrm{d}V_g
= \int_{\sigma(B_\delta(p))}(fu_{\beta}^{2^*-1})(\sigma^{-1}(x))\frac{2^n}{(1+|x|^2)^n}\mathrm{d}x \\
&\; = 2^n(\beta^2-1)^{\frac{n+2}{4}} \int_{\{|x|>\cot\frac{\delta}{2}\}}\frac{f(\sigma^{-1}(x))}{((\beta-1)|x|^2+\beta+1)^{\frac{n+2}{2}}(1+|x|^2)^{\frac{n-2}{2}}}\mathrm{d}x\\
&\;= 2^n\left(\frac{\beta-1}{\beta+1}\right)^{\frac{n-2}{4}}\int_{\mathbb{S}^{n-1}} \int_0^{\sqrt{\frac{\beta+1}{\beta-1}}\tan \frac{\delta}{2}} \frac{f\left(\sigma^{-1}(\sqrt{\frac{\beta+1}{\beta-1}}\frac{\xi}{\rho})\,\rho^{n-1}\right)}{(\rho^2+1)^{\frac{n+2}{2}}(\frac{\beta-1}{\beta+1}\rho^2+1)^{\frac{n-2}{2}}}\mathrm{d}\rho\,\mathrm{d}V_{g_{n-1}},
\end{align*}
where $g_{n-1}$ is the round metric on $\mathbb{S}^{n-1}$. As $n\int_0^\infty\rho^{n-1}(1+\rho^2)^{-\frac{n+2}{2}}\mathrm{d}\rho=1$, we deduce that
$$\lim_{\beta\to 1}\frac{1}{(\beta-1)^{\frac{n-2}{4}}}\int_{B_\delta(p)}fu_{\beta}^{2^*-1}\mathrm{d}V_g=2^{\frac{3n+2}{4}}\,\omega_{n-1}\,\frac{f(p)}{n}.$$
This is the first statement. The second one can be obtained easily, since 
\begin{equation*}
\lim_{\beta\to\infty}\frac{1}{(\beta-1)^{\frac{2+n}{4}}}\int_{\mathbb{S}^n\smallsetminus B_\delta(p)}fu_{\beta}^{2^*-1}\mathrm{d}V_g =2^{\frac{n+2}{4}}\int_{\mathbb{S}^n\smallsetminus B_\delta(p)}f(1-\cos r)^{-\frac{n+2}{4}}\mathrm{d}V_g.
\end{equation*}
This completes the proof.
\end{proof}

\begin{proof}[Proof of Proposition \ref{prop:mu}]
We fix $p\in\mathbb{S}^n$ with $\Gamma_p=\{1\}$. For each $\beta>1$ we set
$$u_{j,\beta}:=u_\beta\circ g_j^{-1},\qquad u_{m+j,\beta}:=u_\beta\circ (\hat{\gamma}g_j)^{-1},\qquad j=1,\ldots,m,$$
and we define
$$w_\beta:=\sum_{j=1}^m(u_{j,\beta}-u_{m+j,\beta}).$$
Since $(A_0)$, $(A_1)$ and $(A_2)$ hold true and $\Gamma_p=\{1\}$, we have that $w_\beta\neq 0$. Hence, there exists $t_\beta\in (0,\infty)$ such that $t_\beta w_\beta\in\mathcal{N}^\phi(\mathbb{S}^n)$, and
\begin{equation*}
c_n^\phi\leq J_n(t_\beta w_\beta)=\frac{1}{n}[Y_n(w_\beta)]^{n/2}, \qquad\text{where } Y_n(u):=\frac{\|u\|^2}{|u|^{2}_{2^*}}
\end{equation*}
and $\|u\|$ and $|u|_{2^*}$ are the norms defined in \eqref{eq:norms}. 

Since $u_{j,\beta}$ solves \eqref{eq:1}, using Lemma \ref{lem:estimate} we estimate
\begin{align*}
\|w_\beta\|^2 &=a_n\sum_{i,j=1}^m \int_{\mathbb{S}^n}[u_{i,\beta}u_{j,\beta}^{2^*-1}+u_{i+m,\beta}u_{j+m,\beta}^{2^*-1}-u_{i,\beta}u_{j+m,\beta}^{2^*-1}-u_{i+m,\beta}u_{j,\beta}^{2^*-1}]\mathrm{d}V_g\\
&= a_n\sum_{i,j=1}^m
\int_{\mathbb{S}^n}(u_{i,\beta}\circ g_{j}+u_{i+m,\beta}\circ\hat{\gamma}g_j-u_{i,\beta}\circ\hat{\gamma}g_j-u_{i+m,\beta}\circ g_{j})u_{\beta}^{2^*-1}\mathrm{d}V_g \\
&= 2ma_n\omega_n+ \frac{2^{n+1}a_n\omega_{n-1}}{n} (\mu_p-\widehat{\mu}_p)(\beta-1)^{\frac{n-2}{2}}+o(\beta-1)^{\frac{n-2}{2}},
\end{align*}
We choose $\delta>0$ such that  $B_\delta(q)\cap B_\delta(q')=\emptyset$ for all points $q,q'\in \Gamma p$ with $q\neq q'$. Then,
\begin{align*}
\int_{\mathbb{S}^n}|w_\beta|^{2^*}\mathrm{d}V_g &\geq \sum_{j=1}^{m}\int_{B_\delta (p_j)}|u_{j,\beta}+\sum_{i\neq j}u_{i,\beta}-\sum_{i}u_{i+m,\beta}|^{2^*}\mathrm{d}V_g\\
&\quad+\sum_{j=1}^{m}\int_{B_\delta (q_j)} |u_{j+m,\beta}+\sum_{i\neq j}u_{i+m,\beta}-\sum_{i}u_{i,\beta}|^{2^*}\mathrm{d}V_g\\
& \geq  2m\int_{B_\delta (p)} u_{\beta}^{2^*}\mathrm{d}V_g + 2^*\int_{B_\delta(p)}\biggl(\sum_{1\leq i\neq j\leq m}u_{i,\beta}\circ g_j + u_{i+m,\beta}\circ\hat{\gamma}g_j \\
&\quad -\sum_{1\leq i,j\leq m}u_{i+k,\beta}\circ g_j + u_{i,\beta}\circ \tilde{\gamma}g_j\biggr) u_\beta^{2^*-1}\mathrm{d}V_g\\
&\geq 2m\omega_n+\frac{2^{n+2}}{n-2}\omega_{n-1}(\mu_p-\widehat{\mu}_p)(\beta-1)^{\frac{n-2}{2}}+o\left((\beta-1)^{\frac{n-2}{4}}\right),
\end{align*}
where we used the inequality $|a+b|^p\geq a^p+pa^{p-1}b$ for $a\geq 0$, $b\in\mathbb{R}$ and $p\geq 1$, and Lemma \ref{lem:estimate}. Thus,
\begin{align*}
|w_\beta|_{2^*}^{-2}\leq &(2ma_n\omega_n)^{\frac{2-n}{n}}-\frac{2^{n+2}a_n^{\frac{2-n}{n}}\omega_{n-1}}{n}(2m\omega_n)^{\frac{2-2n}{n}}(\mu_p-\widehat{\mu}_p)(\beta-1)^{\frac{n-2}{2}}\\
&+o\left((\beta-1)^{\frac{n-2}{4}}\right).
\end{align*}
We conclude that
\begin{equation*}
Y_n(w_\beta)\leq (2ma_n\omega_n)^\frac{2}{n}-C_{n,k}(\mu_p-\widehat{\mu}_p) (\beta-1)^{\frac{n-2}{2}}+o\left((\beta-1)^{\frac{n-2}{4}}\right),
\end{equation*}
where $C_{n,k}:=\frac{2^{n+1}a_n\omega_{n-1}}{n}(2ma_n\omega_n)^{\frac{2-n}{n}}$.

If $\mu_p-\widehat{\mu}_p>0$, then $Y_n(w_\beta)<(2ma_n\omega_n)^\frac{2}{n}=(2mnc_n)^\frac{2}{n}$ for $\beta>1$ sufficiently close to $1$. Therefore,
\begin{equation*}
c_n^\phi\leq J_n(t_\beta w_\beta)=\frac{1}{n}[Y_n(w_\beta)]^{n/2}<2mc_n\qquad\text{for }\beta\text{ sufficiently close to }1.
\end{equation*}
This concludes the proof.
\end{proof}

\section{The proof of the main result}

Next, we compute the sign of $\mu_p - \widehat{\mu}_p$ for $\phi_m:\Gamma_m\to\mathbb{Z}_2$ as in Example \ref{ex:m}. 

\begin{lemma}\label{lem:mu}
Let $\phi_m:\Gamma_m\to\mathbb{Z}_2$ be as in \emph{Example \ref{ex:m}} and let $p=(1,0,0)\in\mathbb{C}\times\mathbb{C}\times\mathbb{R}^{n-3}\equiv\mathbb{R}^{n+1}$. Then, $\mu_p - \widehat{\mu}_p>0$ if and only if $m\geq m_n$, with $m_n$ as in \eqref{eq:m(n)}.
\end{lemma}

\begin{proof}
We have that $p_j=(\mathrm{e}^{2\pi\mathrm{i}j/m},0,0)$ and that $q_j:=\tau p_j$ is orthogonal to $p_i$ for all $i,j=1,\ldots,m-1$. Therefore,
$$\mu_p=m\sum_{j=1}^{m-1}\left(1-\cos\frac{2\pi j}{m}\right)^\frac{2-n}{2}\qquad\text{and}\qquad\widehat{\mu}_p=m^2.$$
As $1-\cos\frac{2\pi j}{m}=2\sin^2(\frac{\pi j}{m})$, we get that
$$\frac{1}{m}(\mu_p-\widehat{\mu}_p)=\sum_{j=1}^{m-1}\left(\frac{1}{\sqrt{2}\sin\frac{\pi j}{m}}\right)^{n-2}-m=:a_{n,m}.$$
If $2\leq m\leq 4$, then $\sqrt{2}\sin\frac{\pi}{m}\geq \sqrt{2}\sin\frac{\pi}{4}=1$. Hence,
\begin{equation*}
a_{n,m}\leq \frac{m-1}{(\sqrt{2}\sin\frac{\pi}{m})^{n-2}}-m\leq -1\quad \forall n\geq 3.
\end{equation*}
If $m\geq 5$, then $\sqrt{2}\sin\frac{\pi}{m}\leq \sqrt{2}\sin\frac{\pi}{5}<1$. Hence,
\begin{equation*}
a_{n,m}>\frac{2}{(\sqrt{2}\sin\frac{\pi}{m})^{n-2}}-m\geq \frac{2}{(\sqrt{2}\sin\frac{\pi}{m})^{n_0}}-m\quad\text{if }n\geq n_0+2.
\end{equation*}
We have that
\begin{equation*}
\frac{2}{(\sqrt{2}\sin\frac{\pi}{m})^{n_0}}-m> 0
\end{equation*}
if, either $n_0=5$ and $m\geq 5$, or $n_0=4$ and $m\geq 6$, or $n_0=3$ and $m\geq 7$. Indeed, setting $x:=\frac{1}{m}$ and $f_{n_0}(x):=(2x)^{1/n_0}-\sqrt{2}\sin(\pi x)$, looking at the graph of $f_{n_0}$ (see Figure \ref{fig}) and computing its value at $\frac{1}{5},\frac{1}{6},\frac{1}{7}$ respectively, we get that
\begin{equation*}
(2x)^{1/n_0}-\sqrt{2}\sin(\pi x)> 0\quad
\begin{cases}
\text{if }n_0=5\text{ and }0\leq x\leq \frac{1}{5},\\
\text{if }n_0=4\text{ and }0\leq x\leq \frac{1}{6},\\
\text{if }n_0=3\text{ and }0\leq x\leq \frac{1}{7}.
\end{cases}
\end{equation*}
On the other hand, we have that
\begin{align*}
a_{5,6} &= \frac{1}{2^{3/2}}\left(\frac{2}{(\sin\frac{\pi}{6})^3}+\frac{2}{(\sin\frac{\pi}{3})^3}+1\right)-6\approx 1.09907,\\
a_{5,5} &= \frac{1}{2^{3/2}}\left(\frac{2}{(\sin\frac{\pi}{5})^3}+\frac{2}{(\sin\frac{2\pi}{5})^3}\right)-5\approx -0.69601,\\
a_{6,5} &= \frac{1}{4}\left(\frac{2}{(\sin\frac{\pi}{5})^4}+\frac{2}{(\sin\frac{2\pi}{5})^4}\right)-5=-\frac{1}{5}.
\end{align*}
This completes the proof for $n\geq 5$.
\begin{figure}[!ht]
\begin{center}
\includegraphics[width=.8\linewidth]{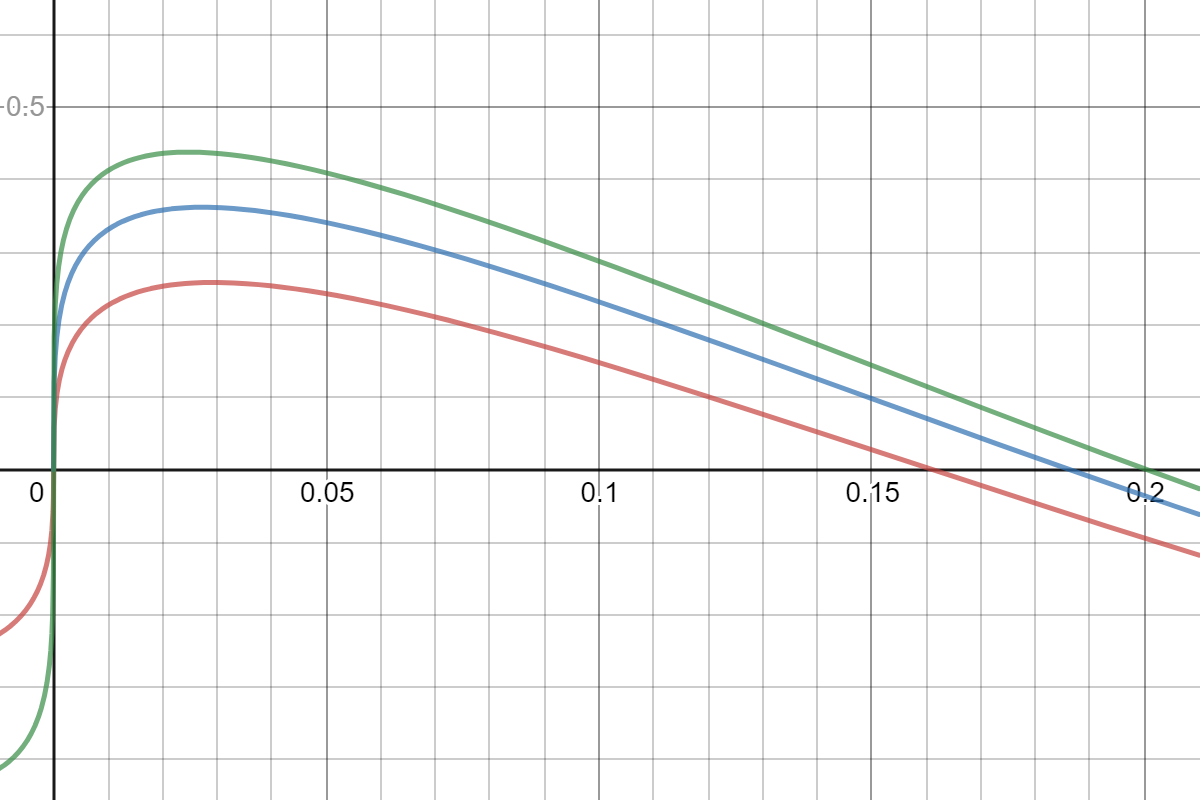}
\end{center}
\caption{The graph of $f_{n_0}$ for $n_0=3,4,5$.}\label{fig}
\end{figure}

If $n=4$, then $a_{4,m}=\frac{1}{6}(m^2-1)-m$. Hence, $a_{4,m}>0$ if and only if $m\geq 7$. 

If $n=3$, direct calculations show that $a_{3,m}<0$ if $m=5,6,7,8$. Note that $t\mapsto\sin t$ is increasing if $t\in [0,\frac{\pi}{2}]$. So, for $m$ even we have that 
\begin{align*}
a_{3,m+1}&=\sum_{j=1}^{m/2}\frac{2}{\sqrt{2}\sin\frac{\pi j}{m+1}}-m-1\geq \sum_{j=2}^{m/2}\frac{2}{\sqrt{2}\sin\frac{\pi j}{m}}-m-1+\frac{2}{\sqrt{2}\sin\frac{\pi}{m+1}}\\
&=\sum_{j=1}^{\frac{m}{2}-1}\frac{2}{\sqrt{2}\sin\frac{\pi j}{m}}+\frac{2}{\sqrt{2}}-m-1+\frac{2}{\sqrt{2}\sin\frac{\pi}{m+1}}-\frac{2}{\sqrt{2}\sin\frac{\pi}{m}}\\
&=a_{3,m}+\frac{2}{\sqrt{2}}\left(\frac{1}{2}-\frac{\sqrt{2}}{2}+\frac{1}{\sin\frac{\pi}{m+1}}-\frac{1}{\sin\frac{\pi}{m}}\right).
\end{align*} 
A similar computation shows that, also for $m$ odd, 
$$a_{3,m+1}\geq a_{3,m}+\frac{2}{\sqrt{2}}\left(\frac{1}{2}-\frac{\sqrt{2}}{2}+\frac{1}{\sin\frac{\pi}{m+1}}-\frac{1}{\sin\frac{\pi}{m}}\right).$$
We claim that
\begin{equation}\label{eq:claim}
\frac{1}{\sin\frac{\pi}{m+1}}-\frac{1}{\sin\frac{\pi}{m}}>\frac{\sqrt{2}-1}{2}\qquad\forall m\geq 9.
\end{equation}
If this is true, then $a_{3,m}>0$ for all $m\geq 9$, and the proof of the lemma is complete. To prove \eqref{eq:claim} note that, since $\frac{t(6-t^2)}{6}=t-\frac{t^3}{6}\leq \sin t\leq t$,
\begin{align*}
\frac{1}{\sin\frac{\pi}{m+1}}-\frac{1}{\sin\frac{\pi}{m}}&\geq \frac{m+1}{\pi}-\frac{6}{\frac{\pi}{m}(6-(\frac{\pi}{m})^2}\\
&=\frac{1}{\pi}-\left(\frac{\frac{\pi}{m}}{6-(\frac{\pi}{m})^2}\right)\geq \frac{1}{\pi}-\left(\frac{\frac{\pi}{9}}{6-(\frac{\pi}{9})^2}\right)\qquad\forall m\geq 9.
\end{align*}
A direct calculation gives 
$$\frac{\frac{\pi}{9}}{6-(\frac{\pi}{9})^2}\approx 0.059383<0.111203\approx \frac{1}{\pi}-\frac{\sqrt{2}-1}{2},$$
which yields \eqref{eq:claim}.
\end{proof}

\begin{proof}[Proof of Theorem \ref{thm:main}]
It follows from Lemma \ref{lem:mu}, Proposition \ref{prop:mu} and Corollary \ref{cor:existence}.
\end{proof}

To conclude, we show that our solutions are different from those of Ding \cite{d}. We write $\mathbb{R}^{n+1}\equiv\mathbb{C}\times\mathbb{R}^{k-2}\times\mathbb{C}\times\mathbb{R}^{m-2}$ with $k,m\geq 2$ and $k+m=n+1$ and, accordingly, we write the points in $\mathbb{R}^{n+1}$ as $(z_1,x_1,z_2,x_2)$.

\begin{proposition} \label{prop:noding}
Let $n>3$. If $u:\mathbb{S}^n\to\mathbb{R}$ is $[O(k)\times O(m)]$-invariant and
$$u(z_1,x_1,z_2,x_2)=-u(-\bar{z}_2,x_1,\bar{z}_1,x_2)\qquad\forall (z_1,x_1,z_2,x_2)\in\mathbb{R}^{n+1},$$ then $u\equiv 0$.
\end{proposition}

\begin{proof}
Without loss of generality, we may assume that $k\leq m$. Since $u$ is $[O(k)\times O(m)]$-invariant it can be  written as
$$u(z_1,x_1,z_2,x_2)=w(|(z_1,x_1)|,|(z_2,x_2)|).$$
Then, for every $(z_1,x_1,z_2,x_2)\in\mathbb{S}^n$, we have that 
$$w(|(z_1,x_1)|,|(z_2,x_2)|)=-w(|(z_2,x_1)|,|(z_1,x_2)|)$$
and, taking $z_1=z_2=0$, we get that 
$$w(|x_1|,|x_2|)=-w(|x_1|,|x_2|)\qquad\forall (x_1,x_2)\in\mathbb{R}^{k-2}\times\mathbb{R}^{m-2}.$$
If $k>2$, this implies that $w\equiv 0$.

On the other hand, if $k=2$ then $m>2$ and, taking $z_1=z_2=0$, we get that $w(0,1)=-w(0,1)=0$. Setting $z_1=0$ we conclude that 
$$0=w(0,1)=-w(|z_2|,|x_2|)\qquad\forall (z_2,x_2)\in\mathbb{C}\times\mathbb{R}^{m-2}.$$
Hence, $w\equiv 0$. 
\end{proof}

 \vspace{15pt}

\begin{flushleft}
\textbf{Mónica Clapp}\\
Instituto de Matemáticas\\
Universidad Nacional Autónoma de México\\
Circuito Exterior, Ciudad Universitaria\\
04510 Coyoacán, Ciudad de México \\
Mexico\\
\texttt{monica.clapp@im.unam.mx} \vspace{10pt}

\textbf{Angela Pistoia}\\
Dipartimento SBAI\\
La Sapienza Università di Roma\\
Via Antonio Scarpa 16 \\
00161 Roma \\
Italy\\
\texttt{angela.pistoia@uniroma1.it}\vspace{10pt}

\textbf{Tobias Weth}\\
Institut für Mathematik \\
Goethe-Universität Frankfurt \\
Robert-Mayer-Str. 10 \\
D-60629 Frankfurt am Main \\
Germany \\
\texttt{weth@math.uni-frankfurt.de}

\end{flushleft}

\end{document}